\documentclass[12pt]{amsart}

\usepackage{amsmath,amssymb,amsthm}
\usepackage{amsfonts}
\usepackage[mathscr]{eucal}
\newtheorem{theorem}{Theorem}[section]

\newtheorem{cor}[theorem]{Corollary}
\newtheorem{corollary}[theorem]{Corollary}

\theoremstyle{definition}

\theoremstyle{remark}
\newtheorem{remark}[theorem]{Remark}

\numberwithin{equation}{section}

\newcommand{\RR}[1]{\mathbb{#1}}

\newcommand{\rd}{{\mathbb R^d}}

\newcommand{\rr}{{\mathbb R}}

\newcommand{\eqd}{\sim}

\begin{document}

\title{\bf Brownian subordinators and fractional Cauchy problems}

\author{Boris Baeumer}
\address{Boris Baeumer, Department of Mathematics and Statistics, University of Otago, PO. Box 56,
Dunedin, NZ}
\email{bbaeumer@maths.otago.ac.nz}

\author{Mark M. Meerschaert}
\address{Mark M. Meerschaert, Department of Probability and Statistics,
Michigan State University, East Lansing, MI 48823}
\email{mcubed@stt.msu.edu}
\urladdr{http://www.stt.msu.edu/$\sim$mcubed/}
\thanks{MMM was partially supported by NSF grant DMS-0417869.}

\author{Erkan Nane}
\address{Erkan Nane, Department of Probability and Statistics,
Michigan State University, East Lansing, MI 48823}
\email{nane@stt.msu.edu}

\begin{abstract}
A Brownian time process is a Markov process subordinated to the
absolute value of an independent one-dimensional Brownian motion.
Its transition densities solve an initial value problem involving
the square of the generator of the original Markov process.  An
apparently unrelated class of processes, emerging as the scaling
limits of continuous time random walks, involve subordination to the
inverse or hitting time process of a classical stable subordinator.
The resulting densities solve fractional Cauchy problems, an
extension that involves fractional derivatives in time.  In this
paper, we will show a close and unexpected connection between these
two classes of processes, and consequently, an equivalence between
these two families of partial differential equations.
\end{abstract}


%\textbf{Mathematics Subject Classification (2000):} 60J65, 60K99.
\maketitle

\section{Introduction}

The goal of this paper is to establish a connection between two
seemingly disparate classes of subordinated stochastic processes,
and their governing equations.  The first class of processes uses
Brownian motion as a subordinator
\cite{allouba2,allouba1,burdzy1,burdzy2,burdzy-khos,bukh}.  The
second class uses the inverse or hitting time process of a
nondecreasing stable L\'evy process as a subordinator
\cite{Zsolution, couplePR, limitCTRW, coupleCTRW, ultraslow2}.  One
application of the Brownian subordinator involves scaling limits of
transverse diffusion in a crack \cite{bukh}.  Inverse L\'evy
subordinators occur in scaling limits of continuous time random
walks used to model anomalous diffusion in physics \cite{MetzlerK,
MetzlerKlafter, SKW}, finance \cite{GMSR, coupleEcon, scalas1}, and
hydrology \cite{WRRapply, BensonTiPM, Eulerian}.  Neither class of
processes is Markovian.  However, their transition densities satisfy
certain abstract differential equations that characterize the
evolution of the process.  Given a Markov process whose transition
densities solve a Cauchy problem, a Brownian subordinator yields
another process whose one dimensional distributions solve a related
initial value problem involving the square of the generator of the
Markov semigroup.  Subordinating the same original Markov process to
an inverse stable L\'evy process leads to a fractional Cauchy
problem, where the integer time derivative is replaced by a
fractional derivative whose order equals the stable index.  Both
situations lead to anomalous subdiffusion, where probability mass
spreads slower than the classical $t^{1/2}$ rate seen in Brownian
motion.  In this paper we will show that, when the order of the
fractional derivative equals 1/2, the two governing equations have
the same unique solution, and hence the two processes have the same
one dimensional distributions.  Similar observations apply for other
values of the fractional time derivative, and there are connections
to subordination by a symmetric stable process.  In applications to
physics, the principal appeal of both classes is that the
subordinators have a pleasant scaling property, that can be
inherited by the subordinated processes.

A L\'evy process $B(t)$ has stationary, independent increments  and
$t\mapsto B(t)$ is continuous in probability, see for example
Bertoin \cite{bertoin} or Sato \cite{sato}.  Classical Brownian
motion in $\rd$ is the special case where $B(t)$ has a mean zero
normal distribution.  In this case, the probability distribution
$p(x,t)$ of $x=B(t)$ solves the diffusion equation $\partial
p/\partial t= D\, \Delta p$ for some $D>0$, where $\Delta=\sum_i
\partial^2/\partial x_i^2$.  The random walk model of particle
motion, where particles undergo independent identically
distributed jumps at regularly spaced interval in time, converges
to Brownian motion as both the spatial and time scales increase to
infinity. This close connection between Brownian motion and the
diffusion equation is due to Bachelier \cite{Bachelier} and
Einstein \cite{Einstein}.  Sokolov and Klafter
\cite{SokolovKlafter} discuss modern extensions to include heavy
tailed particle jumps and random waiting times, leading to
fractional diffusion equations, the simplest of which is
$\partial^\beta p/\partial t^\beta= -D\, (-\Delta)^{\alpha/2} p$.
The fractional Laplacian reflects heavy tailed particle jumps,
where the probability of a jump magnitude exceeding some large
$r>0$ diminishes like $r^{-\alpha}$, and the fractional time
derivative codes heavy tailed waiting times, where the probability
of waiting longer than $t$ falls off like $t^{-\beta}$.  Heavy
tails in space lead to superdiffusion, where a plume of particles
spreads faster than the classical $t^{1/2}$ rate associated with
Brownian motion.  Heavy tails in time lead to subdiffusion, since
long waiting times retard particle motion.  See Metzler and
Klafter \cite{MetzlerK,MetzlerKlafter} for a recent survey.

A completely different model of particle motion emerges from  the
theory of random walks in random media, pioneered by Sinai
\cite{Sinai}.  These models lead to interesting localization
phenomena \cite{FIN, Prosen} not seen in the case of random waiting
times \cite{Bisquert}.  A closely related problem is diffusion on
fractals \cite{Barlow90,BarlowBass,BarlowParkins}.  One particular
application, diffusion in a crack, led Burdzy and Khoshnevisan
\cite{bukh} to consider Brownian subordinators.  In particular, a
two-sided Brownian motion subordinated to an independent Brownian
subordinator yields the scaling limit of transverse diffusion in a
crack.  Essentially, the crack is modeled using the sample path of
the outer process, and the inner process or subordinator codes the
subdiffusive effect of restricted motion in the crack.  Given this
history, it would be difficult to imagine that these two classes of
non-Markovian processes exhibit any close connections.  However, we
will show that they are very closely related.

\section{Background}

Let $B(t)$ be a Brownian motion on $\rd$ and $Y_{t}$ another
independent one-dimensional Brownian motion.  Allouba and Zheng
\cite{allouba2,allouba1} introduce a process they call Brownian time
Brownian motion (BTBM) defined as $Z_t=B(|Y_t|)$.  A related process
called iterated Brownian motion process (IBM) was first considered
by Burdzy \cite{burdzy1}.  Take $B_1(t)$, $B_2(t)$ to be independent
Brownian motions on $\rd$, and $Y_{t}$ as before.  Define a
two-sided Brownian motion by $B(t)=B_1(t)$ for $t\geq 0$ and
$B_2(-t)$ for $t<0$.  Then the IBM process is defined by
$Z_t=B(Y_{t})$.  Various extensions of the BTBM or IBM have been
considered, including a general Markov outer process
\cite{allouba2,allouba1} and a symmetric stable subordinator studied
by Nane \cite{nane3, nane4}.  Similarly, the excursion-based
Brownian-time process is defined by breaking up the path of $|Y(t)|$
into excursion intervals (i.e., maximal intervals of time on which
$|Y(t)|>0$) and, on each such interval, choosing an independent copy
of the Markov process from a finite or an infinite collection.  A
simple conditioning argument shows that all these Brownian time
processes have the same transition density functions, and hence the
same one-dimensional distributions \cite{allouba1}.  A similar
remark holds for symmetric stable subordinators.  That these
processes are non-Markovian can be seen by noting that the
transition densities do not solve the Kolmogorov equations.

Allouba and Zheng \cite{allouba1} and DeBlassie  \cite{deblassie}
show that for iterated Brownian motion (IBM) $Z_t=B(Y_{t})$ the
function
$$
u(t,x)=E_{x}[f(Z_{t})]:=E[f(Z_{t})|Z_0=x]
$$
solves the initial value problem
\begin{equation}\label{ibm-pde0}
\frac{\partial}{\partial t}u(t,x) =
\frac{{\Delta}f(x)}{\sqrt{\pi t}}+ {\Delta}^{2}u(t,x); \quad u(0,x) = f(x)
\end{equation}
for $t>0$ and $x\in \rd$.  The non-Markovian property of IBM is reflected in the appearance
of the time-variable initial term in the PDE.
%In the above PDE we used the following formula for the transition density of Brownian motion in $\RR{R}^{n}$
%$$
%p(t,x,y)=\frac{1}{(4\pi t)^{n/2}}\exp(-\frac{|x-y|^2}{4t})
%$$
%So that it satisfies the PDE:
%$$
%\frac{\partial}{\partial t}p(t,x,y)=\Delta p(t,x,y).
%$$

For a Markov process $X$, the family of linear operators
$T(t)f(x)=E_x[f(X(t))]=E[f(X(t))|X(0)=x]$ forms a bounded
continuous semigroup
 on the Banach space $L^1(\rd)$, and the generator $L_x f(x)=\lim_{h\downarrow 0}
  h^{-1}(T(h)f(x)-f(x))$ is defined on a dense subset of that space, see for example \cite{ABHN,HiPh}.  Then
$u(t,x)=T(t)f(x)$ solves the Cauchy problem
\begin{equation}\label{semigroup-0}
\frac{\partial}{\partial t}u(t,x) =L_xf(x); \quad u(0,x) = f(x)
\end{equation}
for $t>0$ and $x\in \rd$.
Allouba and Zheng \cite{allouba1}  show that if we replace the
outer process $B(t)$  with a continuous Markov process $X(t)$, the
same result holds, except that we replace the Laplacian in the PDE
\eqref{ibm-pde0} with the generator $L_x$ of the continuous
semigroup associated with this Markov process. That is, $u(t,x)=E_{x}[f(Z_{t})]$
solves the initial value problem
\begin{equation}\label{ibm-pde00}
\frac{\partial}{\partial t}u(t,x) =
\frac{{L_x}f(x)}{\sqrt{\pi t}}+ {L_x}^{2}u(t,x); \quad u(0,x) = f(x)
\end{equation}
for $t>0$ and $x\in \rd$.

\begin{remark}
Iterated Brownian motion (IBM) has many properties analogous to
those of Brownian motion.   The process $Z_t$ has stationary  (but
not independent) increments, and is self-similar with index $1/4$,
meaning that $Z_{ct}$ and $c^{1/4}Z_t$ have the same finite
dimensional distributions for every $c>0$.  Burdzy (1993)
\cite{burdzy1} showed that IBM satisfies a Law of the iterated
logarithm (LIL):
$$
\limsup_{t\to\infty}\frac{Z_t}{t^{1/4}(\log \log
(1/t))^{3/4}}=\frac{2^{5/4}}{3^{3/4}} \ \ \  a.s.
$$
A Chung-type LIL by was established by Khoshnevisan  and Lewis
\cite{klewis} and Hu et al. \cite{hu}.  Khoshnevisan and Lewis
\cite{koslew} extended results of Burdzy \cite{burdzy2} to develop a
stochastic calculus for IBM. Local times of IBM were studied by
Burdzy and Khosnevisan \cite{burdzy-khos},  Cs\'{a}ki et al.
\cite{CCFR}, Shi and Yor \cite{shi-yor}, Xiao \cite{xiao}, and Hu
\cite{hu-2}.  Ba\~{n}uelos and DeBlassie \cite{bandeb} studied the
distribution of exit place for IBM in cones.  Nane studied the
lifetime asymptotics of IBM on bounded and unbounded domains
\cite{nane,nane2, nane6}, and generalized isoperimetric-type
inequalities to IBM \cite{nane5}.
\end{remark}

\begin{remark}
Funaki \cite{funaki} considered a different version of iterated
Brownian motion, and studied the PDE connection of that process.
Nane \cite{nane3} established the PDE connection of a related class
of processes defined in a similar manner, replacing the inner time process with a symmetric stable process.
\end{remark}

Zaslavsky \cite{zaslavsky} introduced the fractional kinetic
equation
\begin{equation}\label{frac-derivative-0}
\frac{\partial^\beta}{\partial t^\beta}u(t,x) =
{L_x}u(t,x); \quad u(0,x) = f(x)
\end{equation}
for Hamiltonian chaos, where $0< \beta <1$ and $L_x$ is the
generator of some continuous Markov process $X_0(t)$ started at
$x=0$. Here $\partial^{\beta} g(t)/\partial t^\beta $ is the Caputo
fractional derivative in time, which can be defined as the inverse
Laplace transform of $s^{\beta}\tilde{g}(s)-s^{\beta -1}g(0)$, with
$\tilde{g}(s)=\int_{0}^{\infty}e^{-st}g(t)dt$ the usual Laplace
transform. Baeumer and Meerschaert \cite{fracCauchy} and Meerschaert
and Scheffler \cite{limitCTRW} show that the fractional Cauchy
problem (\ref{frac-derivative-0}) is related to a certain class of
subordinated stochastic processes.  Take $D_t$ to be the stable
subordinator, a L\'evy process with strictly increasing sample paths
such that $E[e^{-sD_t}]=e^{-ts^\beta}$, see for example
\cite{bertoin,sato}.  Define the inverse or hitting time or first
passage time process
\begin{equation}\label{Edef}
E_t=\inf\{x>0: D(x)>t\} .
\end{equation}
The subordinated process $Z_t=X_0(E_{t})$ occurs as the scaling
limit of a continuous time random walk (also called a renewal reward
process), in which iid random jumps are separated by iid positive
waiting times \cite{limitCTRW}. If the waiting times $J_i$ satisfy
$P(J_i>t)=t^{-\beta}L(t)$ where $0<\beta<1$ and $L(t)$ is slowly
varying, then they belong to the strict domain of attraction of a
stable law with index $\beta$, and their partial sum process
$S_t=J_1+\cdots+J_{[t]}$ converges after rescaling to the process
$D_t$, in the Skorokhod $J_1$ topology on $D(\rr^+,\rr)$. The number
of jumps by time $t>0$ is given by the inverse process
$N_t=\max\{n:S_n\leq t\}$ and \cite{limitCTRW} shows that a rescaled
version of $N_t$ converges to the hitting time process $E_t$ in the
same topology.  Similarly, if the iid vector particle jumps $Y_i$
independent of the waiting times satisfy a multivariable regular
variation condition (or if they have a finite covariance matrix),
then the partial sum process $V(t)=Y_1+\cdots+Y_{[t]}$ converges
after linear operator rescaling to the operator L\'evy motion
$X_0(t)$, in the Skorokhod $J_1$ topology on $D(\rr^+,\rd)$, see
\cite[Theorem 4.1]{limitCTRW}.  An operator L\'evy motion is a
L\'evy process such that $X_0(t)$ has a centered operator stable
distribution, see \cite[Example 11.2.18]{RVbook}. This means that
$X_0(ct)$ and $c^E X_0(t)$ are identically distributed for all
$c>0$, for some linear operator $E$, see for example
\cite{JurekMason,RVbook}.  The continuous time random walk $V(S_t)$
models the location of a particle at time $t>0$.  A continuous
mapping argument, under some mild technical conditions, yields
convergence of the rescaled CTRW to the subordinated process
$X_0(E_t)$ in the somewhat weaker $M_1$ topology, see
\cite{limitCTRW}. Given a Banach space and a bounded continuous
semigroup $T(t)$ on that space with generator $L_x$, it is well
known that $p(t,x)=T(t)f(x)$ is the unique solution to the abstract
Cauchy problem
\begin{equation}\label{ACP}
\frac{\partial}{\partial t}p(t,x) =L_x p(t,x); \quad p(0,x) = f(x)
\end{equation}
for any $f$ in the domain of $L_x$ see for example \cite{ABHN,pazy}.
Theorem 3.1 in \cite{fracCauchy} shows that, in this case, the formula
\begin{equation}\label{ACPsoln}
u(t,x) =\int_0^\infty p((t/s)^\beta,x) g_\beta(s)\,ds
\end{equation}
yields the unique strong solution to the fractional Cauchy problem
\eqref{frac-derivative-0}.  Here $g_\beta(t)$ is the smooth
density of the stable subordinator, such that the Laplace
transform $\tilde g_\beta(s)=\int_0^\infty
e^{-st}g_\beta(t)\,dt=e^{-s^\beta}$.  Choose $x\in\rd$ and let
$X(t)=x+X_0(t)$.  It follows from Theorem 5.1 in \cite{limitCTRW}
that, in the special case where $T(t)f(x)=E_x[f(X(t))]$ is the
semigroup on $L_1(\rd)$ associated with the operator L\'evy motion
$X(t)$, the same formula \eqref{ACPsoln} also equals
$u(t,x)=E_x[f(Z_t)]$ where $Z_t=X(E_t)$ is the CTRW scaling limit
process.  Hence the subordinated process $Z_t$ is the stochastic
solution to the fractional Cauchy problem
\eqref{frac-derivative-0} in this case.

In the case where $X_0(t)$ is a L\'evy process started at zero and
$X(t)=x+X_0(t)$ for $x\in\rd$, the generator $L_x$ of the
semigroup $T(t)f(x)=E_x[f(X(t))]$ is a pseudo-differential
operator \cite{applebaum,Jacob,schilling} that can be explicitly
computed by inverting the L\'evy representation. The L\'{e}vy
process $X_0(t)$ has characteristic function $$E[\exp(ik\cdot
X_0(t))]=\exp(t\psi(k))$$ with
$$
\psi (k)=ik\cdot a-\frac{1}{2}k\cdot Qk+ \int_{y\neq 0}\left(
e^{ik\cdot y}-1-\frac{ik\cdot y}{1+||y||^{2}}\right)\phi(dy),
$$
where $a\in \RR{R}^{d}$, $Q$ is a nonnegative definite matrix, and
$\phi$ is a $\sigma$-finite Borel measure on $\RR{R}^{d}$ such
that
$$
\int_{y\neq 0}\min \{1,||y||^{2}\}\phi(dy)<\infty,
$$
see for example \cite[Theorem 3.1.11]{RVbook} and \cite[Theorem
1.2.14]{applebaum}. Let
$$\hat{f}(k)=\int_{\RR{R}^{d}}e^{-ik\cdot
x}f( x)\,dx$$
denote the Fourier transform. Theorem 3.1 in
\cite{fracCauchy} shows that $L_x f(x)$ is the inverse Fourier
transform of $\psi(k)\hat{f}(k)$ for all $f\in D(L_x)$, where
$$
D(L_x)=\{ f\in L^{1}(\RR{R}^{d}):\ \psi(k)\hat{f}(k)=\hat{h}(k)\
\exists \ h \in L^{1}(\RR{R}^{d})  \},
$$
and
\begin{equation}\begin{split}\label{pseudoDO}
L_x f(x)&= a\cdot\nabla f(x) +\frac{1}{2}\nabla \cdot Q\nabla
f(x)\\
&+ \int_{y \neq  0} \left(  f(x+y)-f(x)-\frac{\nabla
f(x)\cdot y}{1+y^{2}} \right)\phi(dy)
\end{split}\end{equation}
for all $f\in W^{2,1}(\RR{R}^{d})$, the Sobolev space of
$L^{1}$-functions whose first and  second partial derivatives are
all $L^{1}$-functions.  This includes the special case where
$X_0(t)$ is an operator L\'evy motion.  We can also write
$L_x=\psi(-i\nabla)$ where $\nabla=(\partial/\partial
x_1,\ldots,\partial/\partial x_d)'$.  For example, if $X_0(t)$ is
spherically symmetric stable then $\psi(k)=-D\|k\|^\alpha$ and
$L_x=-D(-\Delta)^{\alpha/2}$, a fractional derivative in space,
using the correspondence $k_{j}\to -i\partial/\partial x_j$ for
$1\leq j\leq d$.  If $X_0$ has independent stable marginals, then
one possible form is $\psi(k)=D\sum_j (ik_j)^{\alpha_j}$ and
$L_x=D\sum_j \partial^{\alpha_j}/\partial x_j^\alpha$ using
Riemann-Liouville fractional derivatives in each variable.  This
form does not coincide with the fractional Laplacian unless all
$\alpha_j=2$.

\begin{remark}
The literature on fractional calculus uses a different semigroup definition. For example, in \cite{fracCauchy} the
semigroup associated with a L\'evy process $X_0(t)$ started at
$x=0$ is defined by
\[T_{FC}(t)f(x)=E[f(x-X_0(t))]=\int_\rd f(x-y)P_{X_0(t)}(dy) .\]
One physical interpretation of this formula, when $f(x)$ is the probability
density of a random variable $W$, is that $W$ represents the
location of a randomly selected particle at time $t=0$, and
$T_{FC}(t)f(x)$ is the probability density of the random particle
location $W+X_0(t)$ at time $t>0$.  In this paper, we use the
semigroup definition from the Markov process literature, based on
the process $X(t)=x+X_0(t)$, so that
\[T(t)f(x)=E_x[f(X(t))]=\int_\rd f(x+y)P_{X_0(t)}(dy) .\]
Clearly this is just a matter of replacing $X_0$ by $-X_0$.   The
paper \cite{fracCauchy} also uses a different definition $\hat
f_{FC}(k)=\int e^{ik\cdot x} f(x)\,dx$ for the Fourier transform.
Each of these two changes implies a change of $k$ to $-k$ in the
formula for the Fourier transform of the semigroup, and hence the
Fourier symbol $\psi(k)$ is the same for both.  However, the
interpretation of the Fourier symbol as a pseudo-differential
operator changes.  In the notation of this paper, the derivative
$\nabla f(x)$ has Fourier transform $(ik)\hat f(k)$ but in the
notation of \cite{fracCauchy} the Fourier transform is $(-ik)\hat
f(k)$.  Hence the generator $L_x=\psi(i\nabla)$ in that paper, which
explains why equation \eqref{pseudoDO} differs from the
corresponding formula (8) in  \cite{fracCauchy}.
\end{remark}

Any Markov process semigroup operator $T(t)f(x)=E_x[f(X(t))]$ is a pseudo-differential operator
\[T(t)f(x)=(2\pi)^{-d}\int_\rd e^{ik\cdot x} \lambda_t(x,k)\hat f(k)\,dk\]
on the space of rapidly decreasing functions \cite[Theorem
1.4]{JSMarkov}.   Under some mild conditions, if the domain $D(L_x)$
of the generator of the extension of this semigroup to the Banach
space of bounded continuous functions contains all smooth functions
with compact support, then one can write
\begin{equation}\begin{split}\label{pseudoDOMarkov}
L_x f(x)&= a(x)\cdot\nabla f(x) +\frac{1}{2}\nabla \cdot Q(x)\nabla
f(x)\\
&+ \int_{y \neq  0} \left(  f(x+y)-f(x)-\frac{\nabla
f(x)\cdot y}{1+y^{2}} \right)\phi(dy,x)
\end{split}\end{equation}
where $a(x)\in \RR{R}^{d}$, $Q(x)$ is a nonnegative definite matrix, and
$\phi(dy,x)$ is a $\sigma$-finite Borel measure on $\RR{R}^{d}$ for each $x\in\rd$ such
that
$$
\int_{y\neq 0}\min \{1,||y||^{2}\}\phi(dy,x)<\infty,
$$
see \cite{Jacob98}.  In this case, the Fourier symbol
$$
\psi (k,x)=ik\cdot a(x)-\frac{1}{2}k\cdot Q(x)k+ \int_{y\neq 0}\left(
e^{ik\cdot y}-1-\frac{ik\cdot y}{1+||y||^{2}}\right)\phi(dy,x),
$$
and $L_x f(x)=\psi(-i\nabla,x)f(x)$ is the inverse Fourier transform of $\psi(k,x)\hat{f}(k)$.

\begin{remark}
A connection between time-fractional equations and Brownian  time
was also noticed by Orsingher and Behgin \cite{OB}, using a very
different approach.  They show that the density functions of
iterated Brownian motion solve the one-dimensional time-fractional
equation $\partial^{1/2} u/\partial t^{1/2}=\partial^2 u/\partial
x^2$ by considering this equation as the end-member of a class of
fractional telegraph equations.
\end{remark}

\section{ Main results}

The results of this section establish a connection between two
seemingly distinct classes of subordinated stochastic processes.
Markov processes are stochastic solutions to the abstract Cauchy
problem \eqref{ACP}.  Brownian subordinators yield stochastic
solutions to an initial value problem \eqref{ibm-pde00} involving
the square of the Markov generator.  Inverse stable subordinators
yield solutions to a fractional Cauchy problem
\eqref{frac-derivative-0} with the same Markov generator.  In the
case where the stable index $\beta=1/2$, our first results shows
that these two equations have the same solutions, and hence, the
subordinated processes have the same one-dimensional distributions.

\begin{theorem}\label{unified}
Let $L_{x}$ be the generator of a continuous Markov semigroup
$T(t)f(x)=E_{x}[f(X_{t})]$, and take $f\in D(L_{x})$ the domain of
the generator.  Then, both the initial value problem
\eqref{ibm-pde00}, and the fractional Cauchy problem
\eqref{frac-derivative-0} with $\beta=1/2$, have the same solution
given by
\begin{equation}\label{ACPsoln1}\begin{split}
u(t,x) =\frac{2}{\sqrt{4\pi t}}\int_{0}^{\infty}T(s)f(x)\exp\left(-\frac{s^2}{4t}\right)ds .
\end{split}\end{equation}
\end{theorem}

\begin{proof}
Note that the Markov semigroup $T(t)f(x)=E_{x}[f(X_{t})]$ is a
uniformly bounded and strongly continuous semigroup on the Banach
space $L^1(\rd)$.  Then Theorem 3.1 in Baeumer and Meerschaert
\cite{fracCauchy} shows that the function $u(t,x)$ defined by
\eqref{ACPsoln} is the unique solution of the fractional Cauchy
problem (\ref{frac-derivative-0}) for any $f\in D(L_{x})$.  In the
case $\beta=1/2$, we have \cite[Example 1.3.19]{applebaum}
\[g_{1/2} (x)=\frac{1}{\sqrt{4\pi x^3}}\exp\left(-\frac{1}{4x}\right)\]
and then a change of variables shows that
\begin{equation}\label{ACPsoln2}\begin{split}
u(t,x) =\frac{t}{\beta}\int_{0}^{\infty}p(x,s)g_{\beta}(\frac{t}{s^{1/\beta}})s^{-1/\beta -1}ds
=\int_{0}^{\infty}p(x,s)q(t,s)ds
\end{split}\end{equation}
where $p(x,t)=T(t)f(x)$ and
\begin{equation}\begin{split}\label{Ydens}
q(t,s)=2tg_{1/2}({t}/{s^{2}})s^{-3}
&=\frac{2t}{s^{3} \sqrt{4\pi t^3 / s^6}}\exp\left(-\frac{s^2}{4t}\right)\\
&= \frac{2}{\sqrt{4\pi t}}\exp\left(-\frac{s^2}{4t}\right)
\end{split}\end{equation}
is a probability density on $s>0$ for all $t>0$.  Hence we have that \eqref{ACPsoln1}
is the unique solution to the fractional Cauchy problem (\ref{frac-derivative-0}).

Allouba and Zheng \cite{allouba1} show that the initial  value
problem \eqref{ibm-pde00} has solution $u(t,x)=E_{x}[f(Z_{t})]$
where $Z_t=B(|Y_t|)$ is the IBM process.  It is not hard to check
that the function $q(t,s)$ in \eqref{Ydens} is the probability
density of the Brownian subordinator $|Y_t|$.  Then a simple
conditioning argument shows that
\begin{equation}\label{AZsoln}
u(t,x)=E_{x}[f(Z_{t})]=\int_{0}^{\infty}q(t,s)p(x,s)ds
\end{equation}
where $p(x,s)=T(s)f(x)=E_{x}[f(X_{t})]$ is the unique  solution to
the initial value problem \eqref{ACP}.  Hence both the fractional
Cauchy problem (\ref{frac-derivative-0}) and the initial value
problem \eqref{ibm-pde00} have the same solution.
\end{proof}

\begin{corollary}\label{new-ibm-pde}
For $f\in D(\Delta_x)$, both the initial value problem
\begin{equation}\label{ibm-pde3}
\frac{\partial}{\partial t}u(t,x) =
\frac{{\Delta_x}f(x)}{\sqrt{\pi t}}+ \Delta_x^{2}u(t,x); \quad u(0,x) = f(x)
\end{equation}
and the fractional Cauchy problem
\begin{equation}\label{frac-derivative3}
\frac{\partial^{1/2}}{\partial t^{1/2}}u(t,x) =
{\Delta_x}u(t,x); \quad u(0,x) = f(x)
\end{equation}
have the same solution given by \eqref{ACPsoln}, where
\begin{equation}\label{Bdensity}
p(t,x)=T(t)f(x)=\int_\rd f(x-y) (4\pi t)^{-d/2}\exp\left(-\frac {\|y\|^2}{4t}\right)\,dy .
\end{equation}
\end{corollary}

\begin{proof}
Apply Theorem \ref{unified} and note that \eqref{Bdensity} is the density of the Brownian motion Markov process with generator $\Delta_x$.
\end{proof}

Suppose that $X(t)$ is an operator L\'evy motion,
$E_t=\inf\{x>0:D_x>t\}$ is the inverse or hitting time process of
the stable subordinator $D_t$ with $E[e^{-sD_t}]=e^{-ts^\beta}$, and
$Z_t=X(E_t)$.  Let $L_x$ be the generator of the semigroup
$T(t)f(x)=E_x[f(X(t))]$, a pseudo-differential operator defined by
\eqref{pseudoDO}.  It follows from Theorem 5.1 in \cite{limitCTRW}
that, for any initial condition $f\in D(L_{x})$, the function
$u(t,x)=E_x[f(Z_t)]$ solves the fractional Cauchy problem
(\ref{frac-derivative-0}).   Hence we call the non-Markovian process
$Z_t$ the stochastic solution to this abstract partial differential
equation.  The following theorem extends this result to a broader
class of driving processes.

\begin{theorem}\label{Esoln}
For any continuous Markov semigroup $X(t)$ with generator  $L_{x}$,
let $E_t=\inf\{x>0:D_x>t\}$ be the inverse or hitting time process
of the stable subordinator $D_t$, independent of $X$, with
$E[e^{-sD_t}]=e^{-ts^\beta}$ for some $0<\beta<1$, and take
$Z_t=X(E_t)$. Then for any initial condition $f\in D(L_{x})$, the
function $u(t,x)=E_x[f(Z_t)]$ solves the fractional Cauchy problem
\eqref{frac-derivative-0}.
\end{theorem}

\begin{proof}
The proof is very similar to Theorem \ref{unified} above.   Theorem
3.1 in \cite{fracCauchy} shows that the function $u(t,x)$ defined by
\eqref{ACPsoln} is the (unique) solution of the fractional Cauchy
problem (\ref{frac-derivative-0}) for any $f\in D(L_{x})$.  A change
of variables shows that \eqref{ACPsoln2} holds, where $g_\beta$ is
the density of the random variable $D_1$. Corollary 3.1 in
\cite{limitCTRW} shows that the function
$q(t,s)=t\beta^{-1}g_{\beta}(s^{-1/\beta}t)s^{-1/\beta -1}$ is the
density of the hitting time $E_t$, and then the result follows by a
simple conditioning argument.
\end{proof}

\begin{corollary}\label{new-ibm-pde2}
For any continuous Markov process $X(t)$, both the  Brownian-time
subordinated process $X(|Y_t|)$ and the process $X(E_t)$
subordinated to the inverse $1/2$-stable subordinator have the same
one-dimensional distributions.  Hence they are both stochastic
solutions to the fractional Cauchy problem
\eqref{frac-derivative-0}, or equivalently, to the higher order
initial value problem \eqref{ibm-pde00}.
\end{corollary}

\begin{proof}
In the proof of Theorem \ref{unified} we noted that  the function
$q(t,s)$ in \eqref{Ydens} is the probability density of the Brownian
subordinator $s=|Y_t|$ where $Y_t$ is a standard Brownian motion.
The proof of Theorem \ref{Esoln} shows that this function is also
the probability density of the inverse $1/2$-stable subordinator
$s=E_t$.  Then the result follows from a simple conditioning
argument.
\end{proof}

\begin{remark}
The proof of Theorem \ref{unified} relies on Theorem 0.1 in Allouba
and  Zheng \cite{allouba1}, but the proof of Theorem 0.1 in that
paper may not be completely rigorous.  The essential argument, using
integration by parts twice, is that
\begin{equation*}\begin{split}
\frac{\partial}{\partial t}u(t,x)
%&=\frac{\partial}{\partial t}\int_{0}^{\infty}p(s,x)\ q(t,s)ds\\
&=\int_{0}^{\infty}T(s)f(x)\ \frac{\partial}{\partial t}q(t,s)ds\\
&=\int_{0}^{\infty}T(s)f(x)\ \frac{\partial^2}{\partial s^2}q(t,s)ds\\
%&=-\int_{0}^{\infty}\frac{\partial}{\partial s}\left[T(s)f(x)\right]\ \frac{\partial}{\partial s}q(t,s)ds\\
&=q(t,s)\frac{\partial}{\partial s}\left[T(s)f(x)\right]\bigg\vert_{s=0}
+\int_{0}^{\infty}\frac{\partial^2}{\partial s^2}\left[T(s)f(x)\right]q(t,s)ds\\
&=q(t,0)L_x[T(0)f(x)]+\int_{0}^{\infty}L_x^2\left[T(s)f(x)\right]q(t,s)ds\\
&=\frac 1{\sqrt{\pi t}}L_xf(x)+L_x^2\int_{0}^{\infty}T(s)f(x)\ q(t,s)ds\\
\end{split}\end{equation*}
which shows that \eqref{AZsoln} solves the higher order initial
value problem \eqref{ibm-pde00}.  In a later paper \cite{allouba2},
Allouba points out that the last step where the operator $L_x$ is
passed outside the integral is not obvious, and adds this as a
technical condition, which is then verified in the special case
$L_x=\Delta_x$ the Laplacian operator.
\end{remark}

The next result is a restatement of Theorem \ref{unified} for
L\'evy semigroups.  The proof does not use Theorem 0.1 in Allouba
and Zheng \cite{allouba1}, rather it relies on a Laplace-Fourier
transform argument.  We will use the following notation for the
Laplace, Fourier, and Fourier-Laplace transforms (respectively):
\begin{equation*}\begin{split}
\tilde{u}(s,x)&=\int_{0}^{\infty}e^{-st}u(t, x)dt ; \\
\hat{u}(t,k)&=\int_{\RR{R}^{d}}e^{ik\cdot x}u(t, x)dx ; \\
\bar{u}(s,k)&=\int_{\RR{R}^{d}}e^{ik\cdot x}\int_{0}^{\infty}e^{-st}
u(t, x)dtdx .
\end{split}\end{equation*}

\begin{theorem}\label{unified2}
Suppose that $X(t)=x+X_0(t)$ where $X_0(t)$ is a L\'evy process
starting at zero.  If $L_x$ is the generator \eqref{pseudoDO} of the
semigroup $T(t)f(x)=E_x[(f(X_t))]$ on $L^1(\rd)$, then for any $f\in
D(L_{x})$, both the initial value problem \eqref{ibm-pde00}, and the
fractional Cauchy problem \eqref{frac-derivative-0} with
$\beta=1/2$, have the same unique solution given by
\eqref{ACPsoln1}.
\end{theorem}

\begin{proof}
Take Fourier transforms on both sides of \eqref{ibm-pde00} to get
\begin{equation}\label{mmm1}
\frac{\partial \hat{u}(t,k)}{\partial t}=\frac{1}{\sqrt{\pi
t}}\psi (k) \hat{f}(k)+ \psi (k)^{2}\hat{u}(t,k)
\end{equation}
using the fact that $\psi(k)\hat{f}(k)$ is the Fourier transform  of
$L_x f(x)$.  Then take Laplace transforms on both sides to get
\[s\bar{u}(s,k)-\hat{u}(t=0, k)= s^{-1/2}\psi (k) \hat{f}(k)+\psi (k)^{2}\bar{u}(s,k) ,\]
using the well-known Laplace transform formula
$$\int_0^\infty \frac{t^{-\beta}}{\Gamma(1-\beta)}e^{-st}dt=s^{\beta-1}$$
for $\beta<1$.  Since $\hat{u}(t=0, k)=\hat{f}(k)$, collecting
like terms yields
\begin{equation}\label{LFT}
\bar{u}(s,k)=\frac{(1+s^{-1/2}\psi (k))\hat{f}(k)}{s-\psi
(k)^{2}}
\end{equation}
for $s>0$ sufficiently large.

On the other hand, taking Fourier transforms on both sides of
(\ref{frac-derivative-0}) with $\beta=1/2$ gives
\begin{equation}\label{mmm2}
\frac{\partial^{1/2}\hat{u}(t,k)}{\partial t^{1/2}}=\psi(k)\hat{u} (t,k)
\end{equation}
Take Laplace transforms on both sides, using the fact that
$s^{\beta}\tilde g(s)-s^{\beta-1}g(0)$ is the Laplace transform of
the Caputo fractional derivative $\partial^\beta g(t)/\partial
t^\beta$, to get
$$
s^{1/2}\bar{u}(s,k)-s^{-1/2}\hat{f}(k)=\psi(k)\bar{u}(s,k)
$$
and collect terms to obtain
\begin{equation}\label{LFT2}\begin{split}
\bar{u}(s,k)&=\frac{s^{-1/2}\hat{f}(k)}{s^{1/2}-\psi(k)}\\
&= \frac{s^{-1/2}\hat{f}(k)}{s^{1/2}-\psi(k)}\cdot \frac{s^{1/2}+\psi (k)}{s^{1/2}+\psi (k)}\\
&= \frac{(1+s^{-1/2}\psi (k))\hat{f}(k)}{s-\psi
(k)^{2}}
\end{split}\end{equation}
which agrees with \eqref{LFT}.  For any fixed $k\in\rd$, the two
formulas are well-defined and equal for all $s>0$ sufficiently
large.

Theorem 3.1 in Baeumer and Meerschaert \cite{fracCauchy} implies
that the function $u(t,x)$ defined by \eqref{ACPsoln} solves the
fractional Cauchy problem (\ref{frac-derivative-0}) in $L^1(\rd)$
for any $f\in D(L_{x})$.  Take Fourier transforms in \eqref{ACPsoln}
and apply the Fubini theorem to get
\begin{equation}\label{mmm3}
\hat{u}(t,k)=\int_{0}^{\infty}\hat{f}(k)\exp((t/s)^{1/2}\psi(k))
g_{1/2}(s)\,ds .
\end{equation}
Note that $\exp((t/s)^{1/2}\psi(k))$ is bounded since $\psi(k)$ is
negative definite, and then a simple dominated convergence argument
shows that $\hat{u}(t,k)$ is continuous in $t>0$ for any $k\in\rd$.
Hence the uniqueness  theorem for Laplace transforms \cite[Theorem
1.7.3]{ABHN} shows that, for each $k\in\rd$, \eqref{mmm3} is the
unique continuous function whose Laplace transform is given by
either \eqref{LFT} or \eqref{LFT2}.  Since $x\mapsto u(t,x)$ is an
element of $L^1(\rd)$ for every $t>0$, and since two elements of
$L^1(\rd)$ with the same Fourier transform are equal $dx$-almost
everywhere,  \eqref{ACPsoln1} is the unique element of $L^1(\rd)$
whose Fourier transform equals \eqref{mmm3}.

Now if $u(t,x)$ is any solution to \eqref{ibm-pde00}, then it  has
Fourier-Laplace transform $\bar u(s,k)$ given by \eqref{LFT} or
equivalently by \eqref{LFT2}.  Since $\hat  u(t,k)$ solves
\eqref{mmm1}, it is differentiable in $t>0$ and hence continuous.
Then the above argument shows that this solution is equal to
\eqref{ACPsoln} $dx$-almost everywhere for every $t>0$.  Hence
\eqref{ACPsoln} solves both the higher order Cauchy problem
\eqref{ibm-pde00}, and the fractional Cauchy problem
\eqref{frac-derivative-0} with $\beta=1/2$, for any $f\in D(L_{x})$,
and it is the unique solution as an element of $L^1(\rd)$, i.e., it
is unique a.e.-$dx$ for every $t>0$.  The proof of Theorem
\ref{unified} shows that \eqref{ACPsoln} reduces to \eqref{ACPsoln1}
in this case.
\end{proof}

\begin{remark}
Uniqueness to \eqref{frac-derivative-0} was shown by Bajlekova in
\cite[Corollary 3.2]{Bajlekova2001}; Allouba  and Zheng
\cite{allouba1} did not discuss uniqueness of solutions.
\end{remark}

\begin{remark}
It is reasonable to conjecture that Theorem \ref{unified2} can  be
extended to a sub-class of Markov processes $X(t)$ whose generators
are pseudo-differential operators with negative definite symbols
$\psi(k,x)$.  One technical difficulty is that $\psi(k,x)^n$ is
usually not the Fourier symbol of a semigroup generator for integers
$n>1$.

%\alert{Does $\psi$ generate a Markov process? (BM with absorbing boundary?)}.
\end{remark}

The Fourier-Laplace transform method can be extended to establish
connections between other fractional Cauchy problems and their
corresponding higher-order initial value problems.  The next result
gives one such correspondence.

\begin{theorem}\label{third}
Suppose that $X(t)=x+X_0(t)$ where $X_0(t)$ is a L\'evy process
starting at zero.  If $L_x$ is the generator \eqref{pseudoDO} of the
semigroup $T(t)f(x)=E_x[(f(X_t))]$, then for any $f\in D(L_{x})$,
both the initial value problem
\begin{equation}\begin{split}\label{one-third} \frac{\partial
u(t,x)}{\partial t} & = \frac{t^{-2/3}}{\Gamma(1/3)}L_x
f(x)+\frac{t^{-1/3}}{\Gamma(2/3)}L_x^2 f(x) + L_x^3u(t,x);\quad
u(0,x) =  f(x)
\end{split}\end{equation}
and the fractional Cauchy problem \eqref{frac-derivative-0} with
$\beta=1/3$, have the same unique solution given by \eqref{ACPsoln}
where $g_{1/3}(t)$ is the probability density of the $1/3$-stable
subordinator, so that $\int_0^\infty
e^{-st}g_{1/3}(t)\,dt=e^{-s^{1/3}}$ for all $s>0$.
\end{theorem}

\begin{proof}
The proof is very similar to Theorem \ref{unified2}.  Take Fourier transforms on both sides of \eqref{one-third} to get
\[\frac{\partial \hat u(t,k)}{\partial t}  =
\frac{t^{-2/3}}{\Gamma(1/3)}\psi(k) \hat f(k)+\frac{t^{-1/3}}{\Gamma(2/3)}\psi(k)^2 \hat f(k)
+\psi(k)^3\hat u(t,k) ,\]
then take Laplace transforms on both sides to get
\[s\bar{u}(s,k)-\hat f(k)= s^{-1/3}\psi (k) \hat{f}(k)+s^{-2/3}\psi(k)^2 \hat f(k) +\psi(k)^3 \bar{u}(s,k) \]
and collect terms to obtain
\begin{equation}\label{LFT3}
\bar{u}(s,k)=\frac{(1+s^{-1/3}\psi (k)+s^{-2/3}\psi(k)^2)\hat{f}(k)}{s-\psi(k)^3}
\end{equation}
for $s>0$ sufficiently large.

On the other hand, taking Fourier transforms on both sides of \eqref{frac-derivative-0} with $\beta=1/3$ gives
\[\frac{\partial^\beta}{\partial t^\beta}\hat u(t,k) =
\psi(k)\hat u(t,k) ,\]
and then taking Laplace transforms yields
$$s^{1/3}\bar{u}(s,k)-s^{-2/3}\hat{f}(k)=\psi(k)\bar{u}(s,k) .$$
Collecting terms yields
\begin{equation}\label{LFT4}\begin{split}
\bar{u}(s,k)&=\frac{s^{-2/3}\hat{f}(k)}{s^{1/3}-\psi(k)}\\
&=\frac{s^{-2/3}\hat{f}(k)}{s^{1/3}-\psi(k)}\cdot
\frac{s^{2/3}+s^{1/3}\psi(k)+\psi(k)^2}{s^{2/3}+s^{1/3}\psi(k)+\psi(k)^2}\\
&= \frac{(1+s^{-1/3}\psi (k)+s^{-2/3}\psi(k)^2)\hat{f}(k)}{s-\psi(k)^3}
\end{split}\end{equation}
which agrees with \eqref{LFT3}.  For any fixed $k\in\rd$, the two
formulas are well-defined and equal for all $s>0$ sufficiently
large.  The remainder of the argument is identical to Theorem
\ref{unified2}.
\end{proof}

\begin{cor}
Suppose that $X(t)=x+X_0(t)$ where $X_0(t)$ is a L\'evy process
starting at zero, and let $E_t=\inf\{x>0:D_x>t\}$ be the inverse or
hitting time process of the $1/3$-stable subordinator $D_t$,
independent of $X$, with $E[e^{-sD_t}]=e^{-ts^{1/3}}$.  If $L_x$ is
the generator \eqref{pseudoDO} of the semigroup
$T(t)f(x)=E_x[(f(X_t))]$ and $Z_t=X(E_t)$, then for any initial
condition $f\in D(L_{x})$, the function $u(t,x)=E_x[f(Z_t)]$ solves
the higher order initial value problem \eqref{one-third}.
\end{cor}

\begin{proof}
The corollary follows immediately from Theorem \ref{third} together with Theorem 3.1 in \cite{fracCauchy}.
\end{proof}

\begin{remark}
An easy extension of the argument for Theorem \ref{third} shows
that, under the same conditions, for any $n=2,3,4,\ldots$ both the
initial value problem
\begin{equation}\begin{split}\label{one-thirda}
\frac{\partial u(t,x)}{\partial t} & =\sum_{j=1}^{n-1} \frac{t^{1-j/n}}{\Gamma(j/n)} L_x^j f(x)
+ L_x^n u(t,x);\quad  u(0,x) =  f(x)
\end{split}\end{equation}
and the fractional Cauchy problem \eqref{frac-derivative-0}  with
$\beta=1/n$, have the same unique solution given by \eqref{ACPsoln}
with $\beta=1/n$.  Hence the process $Z_t=X(E_t)$ is also the
stochastic solution to this higher order initial value problem.
\end{remark}

This paper has established a connection between processes
subordinated to a Brownian time subordinator, and to an inverse
L\'evy stable subordinator.  A similar but weaker correspondence can
also be established for stable time subordinators.  Let
$T(t)f(x)=E_x[f(X(t))]$ be the semigroup of a continuous Markov
process $X(t)$ and let $L_x$ be its generator. Take $S_t$ a standard
symmetric stable L\'evy process with index $0<\alpha< 2$, so that
$E[e^{ikS(t)}]=e^{-t|k|^\alpha}$.  Denote by $p(t,x)$ the density of
$S(t)$, and let $Z_t=X(|S_t|)$.  Nane \cite{nane3} shows that if
$\alpha=l/m$, where $l$ and $m$ are relatively prime, then
$u(t,x)=E_x[f(Z_t)]$ solves
\begin{equation*}
(-1)^{l+1}\frac{\partial ^{2m}}{\partial t^{2m}}u(t,x) =
-2\sum_{i=1}^{l} \frac{\partial^{2l-2i}}{\partial
x^{2l-2i}}p(t,x)\bigg\vert_{x=0}L_x^{2i-1}f(x)
 -\ L_x^{2l}u(t,x);\quad
u(0,x) = f(x)
\end{equation*}
for any bounded measurable function $f$ in the domain of $L_x$, with
$D^{\gamma}f$  bounded and H\"{o}lder continuous for all multi index
$\gamma$ such that $|\gamma |=2l$.

Suppose that $X(t)$ is a self-similar process with Hurst index $H$,
so  that $X(ct)\eqd c^H X(t)$ (equality in distribution).  Then it
is not hard to check that $Z_t=X(|S_t|)$ is also self-similar with
$Z_{ct}\eqd c^{H/\alpha}Z_t$.  If $1<\alpha\leq 2$ and we take $E_t$
to be the inverse or hitting time process for a stable subordinator
with index $\beta=1/\alpha$, then $E_{ct}\eqd c^\beta E_t$ and it
follows that the process $X(E_t)$ is self-similar with the same
index as $X(|S_t|)$.  Corollary 3.1 in \cite{limitCTRW} shows that
$E_t$ has moments of all orders, while the mean of $|S_t|^\rho$
diverges for $\rho>\alpha$.  Hence it seems that these two processes
are not equivalent.

Finally, we note that the equivalence established in this paper does not extend to strict subdomains of $\rd$.  Consider the Banach space $X=L^1(\rr^+)$ and the shift semigroup
$\left[T(t)f\right](x):=f(x+t)$ with generator
$L_xf=\frac d{dx}f$ and domain $\mathcal D(L_x)=\{f\in L^1(\rr^+):f'\in
L^1(\rr^+)\}$.  In this case, solutions to the higher order initial value problem
   \begin{equation}\label{hopde}
\frac{\partial}{\partial t}u(t,x) = \frac{{L_x}f(x)}{\sqrt{\pi t}}+
{L_x}^{2}u(t,x); \quad u(0,x) = f(x);x\ge 0
\end{equation}
are not unique.  To see this, note that if $u_1,u_2$ were any two
solutions then $u=u_1-u_2$ would solve
\begin{equation}\label{nonunique}
\frac{\partial}{\partial t}u(t,x) = {L_x}^{2}u(t,x); \quad
u(0,x)=0,x\ge 0
\end{equation}
and uniqueness would require that $u\equiv 0$ is the only solution
to \eqref{nonunique}.  However,
$$u(t,x)=\frac1{\sqrt{4\pi t}}\exp\left(-\frac{(x+1)^2}{4t}\right)$$
solves \eqref{nonunique} as
well, so that solutions to \eqref{hopde} are not unique.  On the other hand, solutions to the corresponding fractional Cauchy problem \eqref{frac-derivative-0} with $\beta=1/2$ are unique \cite{Bajlekova2001}.  Hence the two forms are not equivalent on this domain.  It is an interesting open problem to find governing partial differential equations for Brownian time processes on bounded subdomains.  For a typical Markov process with generator $L_x$ on a bounded domain, one can solve the boundary value problem $\dot u=L_x u$; $u=f$ on the boundary, as the expectation of $X(\tau)$ where $\tau$ is the hitting time at the boundary.  DeBlassie \cite{deblassie} shows that the analogous result does not hold for iterated Brownian motion, and it is likely that a resolution of these problems will require a novel approach.

\end{document}